\documentclass[12pt, A4]{article}
\usepackage{amsthm,amsmath,amssymb,amsfonts,mathrsfs,hyperref,microtype}
\usepackage{graphicx,color}

\usepackage{geometry}
\newcommand{\weq}{\ = \ }

\newcommand{\wle}{\ \le \ }

\newtheorem{definition}{Definition}[section]
\newtheorem{theorem}{Theorem}[section]
\newtheorem{proposition}{Proposition}[section]
\newtheorem{corollary}{Corollary}[section]

\newtheorem{lemma}{Lemma}[section]
\newtheorem{assumption}{Assumption}[section]
\newtheorem{example}{Example}[section]

\def\N{{\rm I\kern-0.16em N}}
\def\R{{\rm I\kern-0.16em R}}
\def\E{{\rm I\kern-0.16em E}}
\def\Pro{{\rm I\kern-0.16em P}}
\def\F{{\rm I\kern-0.16em F}}
\def\B{{\rm I\kern-0.16em B}}
\def\C{{\rm I\kern-0.46em C}}
\def\G{{\rm I\kern-0.50em G}}

\numberwithin{equation}{section}
\font\eka=cmex10
\usepackage{ae}
\def\ind{\mathrel{\hbox{\rlap{%
\hbox to 7.5pt{\hrulefill}}\raise6.6pt\hbox{\eka\char'167}}}}
\begin{document}
\title{\textbf{Stochastic Differential Equations with Discontinuous Diffusions}}
\date{\today}

\renewcommand{\thefootnote}{\fnsymbol{footnote}}

\author{Soledad Torres\footnotemark[1] \, and \, Lauri Viitasaari\footnotemark[2]}

\footnotetext[1]{Universidad de Valpara\'iso, Facultad de Ingenier\'ia, CIMFAV, Chile, {\tt soledad.torres@uv.cl}.}

\footnotetext[2]{Aalto University School of Business, Department of Information and Service Management, Finland (\textbf{Corresponding author}), {\tt lauri.viitasaari@iki.fi}.}

\maketitle

\begin{abstract}
We study one-dimensional stochastic differential equations of form $dX_t = \sigma(X_t)dY_t$, where $Y$ is a suitable H\"older continuous driver such as the fractional Brownian motion $B^H$ with $H>\frac12$. The innovative aspect of the present paper lies in the assumptions on diffusion coefficients $\sigma$ for which we assume very mild conditions. In particular, we allow $\sigma$ to have discontinuities, and as such our results can be applied to study equations with discontinuous diffusions. 
\end{abstract}



\section{Introduction}
The theory of stochastic differential equations (SDEs) is rather well-established whenever coefficients are smooth enough. In order to ensure the existence and uniqueness of solutions to SDEs, most of the assumptions used for the coefficients are related to Lipschitz continuity and/or linear growth. Only very few cases have been studied under more general conditions, especially, when dealing with discontinuous or singular coefficients. 
Nakao \cite{nakao72} proved pathwise uniqueness of solutions to SDEs driven by Brownian motion, assuming diffusion coefficient to be uniformly positive and of bounded variation on compact intervals. Later Engelbert and Schmidt \cite{ensc85-1,ensc85} proved an existence of weak solution to the SDE 
\begin{equation}\label{brow}
d X_t = \sigma(X_t)dB_t
\end{equation}
driven by Brownian motion, where $\sigma$ is a general real-valued measurable function satisfying $\int \frac{1}{\sigma^2(s)} ds < \infty $. However, the uniqueness in law fails in general. In 1983 Le Gall \cite{legall83} extended the results of Nakao 
by proving existence and uniqueness provided that $\sigma$ is bounded below away from zero, and satisfies 
$ |\sigma(x) - \sigma(y)|^ 2 \le  |g(x) - g(y)|$ for some increasing and bounded function $g$.
In \cite{bch05}, the authors studied the existence of strong and positive solutions and pathwise uniqueness in the case  $\sigma(x)=|x|^{\alpha}$, where $\alpha\in(0,1)$. 

Existence and uniqueness of strong solutions of SDEs driven by fractional Brownian motion under Lipschitz and linear growth conditions was proved by Nualart and R$\check{a}\text{\c{s}}$canu \cite{Nualart_Rascanu_2002}. After this seminal paper, such equations are studied by many authors (see, e.g. references in a monograph \cite{Mish}). In the case of the fractional Brownian motion, articles studying discontinuous coefficients are extremely rare. For an SDE 
\begin{equation}\label{sde}
X_t = X_0 + \int_0^t b(s,X_s) ds + \int_0^t ¡\sigma(s,X_s) dB^H_s
\end{equation}
with discontinuities in the drift $b$, we can mention papers \cite{BO} and \cite{MN}. In \cite{BO} the authors proved existence of a weak solution to \eqref{sde} in a case $\sigma \equiv 1$ and $b(s,X_s) = b_1(s, X_s) + b_2(s, X_s)$, where $b_1(s, x)$  is a H\"older continuous function of order strictly larger than $1 - \frac{1}{2H}$  in $x$ and strictly larger than $H- 1/2$ in $t$, and $b_2$ is a real bounded nondecreasing left- or right-continuous function. Similarly in \cite{MN}, the authors applied Girsanov theorem to prove existence of weak solutions in the case of constant $\sigma$ and discontinuous $b$. Finally, we mention \cite{nualart-et-al} where the Lipschitz continuity in $\sigma$ was relaxed. In \cite{nualart-et-al} the authors studied existence of solutions, in a case where $\sigma$ belongs to a class of functions including $\sigma(x)=|x|^\gamma$ with $\gamma \in (0,1)$ as a prototype. 

While the case of discontinuous drift $b$ is studied in the above mentioned articles, to the best of our knowledge there exists only one article by Garz\'on et al. \cite{soledad-et-al} where $\sigma$ in \eqref{sde} is allowed to be discontinuous. In \cite{soledad-et-al} the authors proved existence and uniqueness for a particular equation
\begin{equation}\label{DISC}
X_t = X_0 + \int_0^t \sigma (X_s) d B^H_s \quad , \quad t \geq 0;
\end{equation}
where $\sigma$ is the discontinuous function given by
\begin{equation}\label{sigma1}
\sigma (x) = \frac{1}{\alpha} 1\{x \geq 0\} + \frac{1}{1-\alpha} 1 \{x < 0\}, \ \alpha \in \left(0,\frac12\right).
\end{equation}

In this article we will study  existence and uniqueness for the SDE
\begin{equation}\label{DISC-1}
dX_t = \sigma(X_t) dY_t,
\end{equation}
where $\sigma(x)$ is a general function of locally bounded variation and satisfies $\sigma(x)\geq 0$ or $\sigma(x)\leq 0$, and $Y$ is a H\"older continuous process of order $\alpha>\frac12$, satisfying certain sufficient variability assumption (see Assumption \ref{assumption:key}). Possible driving forces $Y$ include, among others, fractional Brownian motions with $H>\frac12$ and the Rosenblatt process. In particular, our results generalises the results provided in \cite{nualart-et-al} and \cite{soledad-et-al}. To the best of our knowledge, this article is the first attempt towards general theory of one-dimensional SDEs driven by H\"older continuous forces, where we allow discontinuities for $\sigma$. Our results are based on a recent integration theory developed in \cite{chen-et-al}. 

The rest of the paper is organised as follows. In Section \ref{sec:SDE} we state and discuss our main results. In Section \ref{sec:fractional} we recall some basic facts on fractional derivatives and generalised Lebesgue-Stieltjes integrals, and in Section \ref{sec:integration} we build up the integration theory that we need to prove our main results. Section \ref{sec:proofs} is devoted to the proofs of our main results.

\section{Stochastic differential equations with discontinuous coefficients}
\label{sec:SDE}
In this article we consider stochastic differential equations of form
\begin{equation}
\label{eq:SDE}
dX_t = \sigma(X_t)dY_t
\end{equation}
with some (possibly random) initial condition $X_0$ and suitable driving force $Y$ that is H\"older continuous of order $\alpha > \frac12$. The innovative aspect of the present paper lies in the assumptions on the coefficient $\sigma$, that we allow to contain discontinuities. More precisely, throughout the paper we assume merely that $\sigma\geq 0$ ($\leq 0$, respectively) and $\sigma$ is of locally bounded variation such that $\frac{1}{\sigma}$ is locally integrable. This means that we also allow jump-type discontinuities for $\sigma$, which makes the analysis of \eqref{eq:SDE} rather difficult.

Our existence and uniqueness result is based on the Laplace method. Since $\frac{1}{\sigma}$ is locally integrable, the function
\begin{equation}
\label{eq:lambda}
\Lambda(x) = \int_a^x \frac{1}{\sigma(y)}dy
\end{equation}
is well-defined. Furthermore, since $\sigma$ is non-negative, the function $\Lambda$ is increasing, and thus the inverse $\Lambda^{-1}$ exists. We will show that this gives us a solution candidate $\Lambda^{-1}(Y_t + \Lambda(X_0)-Y_0)$. This is in line with the classical results for Lipschitz diffusions $\sigma$. Indeed, since $\alpha>\frac12$, we have
\begin{equation}
\label{eq:ito-intro}
f(Y_t) = f(Y_0) + \int_0^t f'(Y_u)d Y_u
\end{equation}
for all smooth functions $f$. Now if $\sigma$ is Lipschitz, then one easily obtains from the formula above that $\Lambda^{-1}(Y_t)$ is a solution to \eqref{eq:SDE}. Moreover, the uniqueness can be obtained by a certain fixed point argument. However, in our case the following questions arise;
\begin{enumerate}
\item Does the integral $\int_0^t \sigma(\Lambda^{-1}(Y_t))dY_t$ exist for arbitrary bounded variation function $\sigma$?
\item Does Equation \eqref{eq:ito-intro} hold?
\item In what sense the solution is unique?
\end{enumerate} 
Albeit easily stated, the above mentioned questions are rather subtle in the presence of jumps in $\sigma$. First of all, the existence of the pathwise integral is far from clear as usually the discontinuities of $\sigma$ imply that $\sigma(X_t)$ behaves rather badly even when $X_t$ is nice enough. For example, if $\sigma$ has a discontinuity at $x=0$, then $\sigma(X_t)$ can be of bounded $p$-variation for some $p$ only if $X$ crosses the zero-level finitely many times. The latter condition is not satisfied by many interesting random processes such as Brownian motions or fractional Brownian motions. Instead, they both have uncountably many crossings of the zero-level. Secondly, one cannot apply a fixed point argument in a straightforward manner to obtain uniqueness, which follows again from the bad behaviour of $\sigma(X_t)$. In contrast, in the Lipschitz case $\sigma(X_t)$ is H\"older continuous of the same order as $X$, which then can be used to derive some estimates. 

The key to handle bad coefficients $\sigma$ is to compensate its bad behaviour by variability of the driving force $Y$. The heuristic argument is that, while $\sigma$ may have discontinuities, the process $Y$ do not spend time around these points so that discontinuities of $\sigma$ can be handled. This heuristic is encoded into the following assumption.
\begin{assumption}
\label{assumption:key}
Let $\alpha \in \left(\frac12,1\right)$ be such that $X$ is $\alpha$-H\"older continuous. We assume that 
there exists $\beta\in(1-\alpha,\alpha)$ and $\epsilon>0$ such that 
\begin{equation}
\label{eq:sufficient-variability}
\sup_{y\in\R}\E\int_0^T |X_t-y|^{-\frac{\beta+\epsilon}{\alpha}}dt < \infty.
\end{equation}
\end{assumption}
\begin{example}
\label{example:density}
If $X_t$ has a density function $p_t(y)$ Lebesgue almost everywhere that satisfies
\begin{equation}
\label{eq:density}
\sup_{y\in\R} p_t(y) \in L^1([0,T]),
\end{equation}
then \eqref{eq:sufficient-variability} is automatically valid (cf. \cite{chen-et-al}). This class is already very large, and includes many interesting examples. For example, all Gaussian process $X$ with variance function $V(t)$ satisfying $[V(t)]^{-\frac12} \in L^1([0,T])$ belong to this class. In particular, fractional Brownian motion $B^H$ belongs to this class, and thus $B^H$ with $H>\frac12$ satisfies Assumption \ref{assumption:key}. Another interesting example satisfying Assumption \ref{assumption:key} is the Rosenblatt process $Z^H$ of order $H>\frac12$ (see, e.g. \cite{veil-taqqu} and references therein for basic properties of this process). To the best of our knowledge, SDEs driven by the Rosenblatt process are not extensively studied in the literature. Finally, any stationary process with bounded density function satisfies \eqref{assumption:key}. These examples should convince the reader that the class of possible driving noises in \eqref{eq:SDE} is considerably large. For more interesting examples, we refer to Subsection 3.4 of \cite{chen-et-al}.  
\end{example}
For the coefficient $\sigma$ we make the following assumption.
\begin{assumption}
\label{assu:sigma}
We suppose that $\sigma$ is of locally bounded variation. Moreover, we assume that $\sigma(x) \geq 0$ ($\leq 0$, respectively) for all $x \in \R$ and that $\frac{1}{\sigma}$ is locally integrable.
\end{assumption}
The following existence result is the first main theorem of the present paper.
\begin{theorem}
\label{thm:existence}
Suppose that $Z_t = \Lambda(X_0)+Y_t-Y_0$ satisfies Assumption \ref{assumption:key} and that $\sigma$ satisfies Assumption \ref{assu:sigma}. Then \eqref{eq:SDE} admits a solution that is given by $X_t=\Lambda^{-1}(\Lambda(X_0)+Y_t-Y_0)$, where $\Lambda^{-1}$ is the inverse of $\Lambda$ given by \eqref{eq:lambda}.
\end{theorem}
As the coefficient $\sigma$ has rather bad behaviour, one cannot expect general uniqueness result. However, we can provide the following partial answer. 
\begin{theorem}
\label{thm:uniqueness}
Suppose $\sigma$ satisfies Assumption \ref{assu:sigma} and let $X$ be an arbitrary solution to SDE \eqref{eq:SDE} satisfying Assumption \ref{assumption:key}. 
Set 
$$
\tau = \inf\{t\geq 0:\sigma(X_t) = 0\}.
$$
Then $\tau$ is uniquely defined, and the solution $X$ is unique on $[0,\tau]$. In particular, if $\sigma(x)\neq 0$ for all $x$, the solution $X$ is unique in the class of processes satisfying Assumption \ref{assumption:key}.
\end{theorem}
By combining these two theorems we obtain the following two corollaries. \begin{corollary}
\label{cor:uniqueness}
Let $\sigma$ satisfy Assumption \ref{assu:sigma} and suppose that $Z_t=\Lambda^{-1}(\Lambda(X_0)+Y_t-Y_0)$ satisfies Assumption \ref{assumption:key}. Then it is the unique solution to \eqref{eq:SDE} on $[0,\tau]$, where 
$
\tau = \inf\{t\geq 0:\sigma(Z_t)=0\}.
$
\end{corollary}
\begin{corollary}
\label{cor:uniqueness2}
Let $\sigma$ satisfy Assumption \ref{assu:sigma} and let $X_0,Y_0\in\R$ be constants. Suppose further that $Y_t$ admits a density function $p_t(y)$ almost everywhere such that 
$$
\sup_{y\in\R} p_t(y) \in L^1([0,T]).
$$
If further $\sigma \geq \epsilon$ for all $x\in\R$, then \eqref{eq:SDE} has a unique solution in the class of processes satisfying \ref{assumption:key} given by 
$Z_t = \Lambda^{-1}(\Lambda(X_0)+Y_t - Y_0)$.
\end{corollary}

\subsection{Examples}
In this subsection we present some interesting examples. Throughout we assume that the driving process $Y_t$ is H\"older continuous of order $\alpha>\frac12$ and has a density $p_t(y)$ satisfying \eqref{eq:density}. Such processes are discussed in Example \ref{example:density}, and include particularly the case of the fractional Brownian motion $B^H$ with $H>\frac12$. We stress also that the following examples are simply illustrations how our results can be applied. For notational simplicity, we also assume $Y_0=0$. 
\begin{example}
Let 
$$
\sigma(x) = \beta_+ \textbf{1}_{x\geq a} + \beta_-\textbf{1}_{x<a} + \sigma_0(x),
$$
where $\sigma_0(x)\geq 0$ is an arbitrary function of locally bounded variation, and $\beta_+,\beta_->0$. Then for any initial condition $X_0$ the SDE \eqref{eq:SDE} admits a unique solution $X_t = \Lambda^{-1}(\Lambda(X_0) + Y_t- Y_0)$. This can be viewed as a generalisation of the results provided in \cite{soledad-et-al} to cover larger class of coefficients $\sigma$ and drivers $Y$. Indeed, the particular equation studied in \cite{soledad-et-al} can be recovered by choices $\sigma_0 \equiv 0$, $\beta_+ = \frac{1}{\alpha}$, $\beta_- = \frac{1}{1-\alpha}$, $a=0$, and $Y = B^H$. 
\end{example}
\begin{example}
\label{example:nualart}
Let $\sigma$ be of locally bounded variation such that $\sigma(x)\geq |x|^{\gamma}$, where $\gamma\in(0,1)$, and suppose $X_0$ is a constant. It follows directly from Theorem \ref{thm:existence} that $X_t = \Lambda^{-1}(\Lambda(X_0)+Y_t)$ provides a solution to \eqref{eq:SDE}. This generalises the existence result provided in \cite{nualart-et-al}, where $\sigma$ was assumed to be monotonic and continuous. In addition, the authors in \cite{nualart-et-al} posed additional condition $\gamma < \frac{1}{H}-1$. In comparison, here we do not need additional assumptions for $\gamma$, or on continuity or monotonicity of $\sigma$. Furthermore, \cite{nualart-et-al} did not discuss the uniqueness in details. Applying Theorem \ref{thm:uniqueness} we can directly say that the solution is unique (in the class of processes satisfying \ref{assumption:key}) up to the first point $\tau$ when $\sigma(X_\tau)=0$. Furthermore, we can apply Corollary \ref{cor:uniqueness} to study when $X_t = \Lambda^{-1}(\Lambda(X_0)+Y_t)$ provides the unique solution. For this $\sigma(x) \geq |x|^\gamma$ implies 
$
|\Lambda^{-1}(y)|^{-\frac{\beta}{\alpha}} \leq |y|^{-\frac{\beta}{\alpha(1-\gamma)}},
$
which leads to a restriction $\gamma < 2- \frac{1}{\alpha}$ by choosing $\beta \approx 1- \alpha$. This means that for $\alpha>\frac23$ we can tackle larger values of $\gamma$ compared to \cite{nualart-et-al}, and beyond continuity or monotonicity assumptions. On the other hand, for $\frac{1}{2}< \alpha <\frac23$ our condition $\gamma < 2-\frac{1}{\alpha}$ is stronger than the one posed in \cite{nualart-et-al}. We also point out that, by using continuity of $\sigma$, the authors of \cite{nualart-et-al} were able to study multidimensional SDEs. In this article we are only studying one-dimensional problems.
\end{example}
\begin{example}
Let $\sigma(x) = \epsilon_0 + f(x)$, where $\epsilon_0>0$ and $f(x)$ is the Cantor function. Furthermore, let $X_0$ be a constant. Then Corollary \ref{cor:uniqueness2} implies that $X_t = \Lambda^{-1}(\Lambda(X_0)+Y_t)$ provides the unique solution to \eqref{eq:SDE}. This provides an example of an SDE that involves the Cantor function and is still solvable uniquely. Even if this SDE does not have direct practical applications, such equations are interesting at least out of academic curiosity.
\end{example}
\section{Fractional integrals and derivatives}
\label{sec:fractional}
In this section we recall some preliminaries on fractional integrals and the concept of generalised Lebesgue-Stieltjes integral. For details we refer to \cite{Nualart_Rascanu_2002,Samko_Kilbas_Marichev_1993, Zahle_1998}.

Throughout this section, let $T < \infty$ be fixed. The fractional left and right Riemann--Liouville integrals of order $\theta > 0$ of a function $f \in L_1$ are denoted by
\[
 I^\theta_{0+}f(t)
 \weq \frac{1}{\Gamma(\theta)} \int_0^t \frac{f(s)}{(t-s)^{1-\theta}} \, ds
\]
and
\[
 I^\theta_{T-}f(t)
 \weq \frac{(-1)^{-\theta}}{\Gamma(\theta)} \int_t^T \frac{f(s)}{(t-s)^{1-\theta}} \, ds.
\]
It is known that the integral operators $I^\theta_{0+}, I^\theta_{T-}: L_1 \to L_1$ are linear and one-to-one. The inverse operators are known as Riemann--Liouville fractional derivatives, and denoted by $I^{-\theta}_{0+} = (I^\theta_{0+})^{-1}$ and $I^{-\theta}_{T-} = (I^\theta_{T-})^{-1}$. Furthermore, it is known that, for any $\theta \in (0,1)$ and for any $f \in I^\theta_{0+}(L_1)$ and $g \in I^\theta_{T-}(L_1)$, the Weyl--Marchaud derivatives 
\[
 D_{0+}^\theta f(t)
 \weq \frac{1}{\Gamma(1-\theta)}\left( \frac{f(t)}{t^\theta} + \theta \int_0^t \frac{f(t)-f(s)}{(t-s)^{\theta+1}} \, ds \right)
\]
and
\[
 D_{T-}^\theta g(t)
 \weq \frac{(-1)^\theta}{\Gamma(1-\theta)}\left( \frac{g(t)}{(T-t)^\theta} + \theta \int_t^T \frac{g(t)-g(s)}{(s-t)^{\theta+1}} \, ds \right)
\]
are well defined, and coincide with the Riemann--Liouville derivatives by relations $D_{0+}^\theta f(t) = I^{-\theta}_{0+} f(t)$ and $D_{T-}^\theta g(t) = I^{-\theta}_{T-} g(t)$ for almost every $t \in (0,T)$.

Let now $f$ and $g$ be functions such that the limits $f(0+), g(0+), g(T-)$ exist in $\R$, and denote $f_{0+}(t) = f(t) - f(0+)$ and $g_{T-}(t) = g(t) - g(T-)$. When $f_{0+} \in I^\theta_{0+}(L_p)$ and $g_{T-} \in I^{1-\theta}_{T-}(L_q)$ for some $\theta \in [0,1]$ and $p,q \in [1,\infty]$ such that $1/p+1/q =1$,
the fractional version of the Stieltjes integral introduced by Z\"ahle \cite{Zahle_1998} is defined by
\begin{equation}
 \label{eq:ZSIntegral}
 \begin{aligned}
 \int_0^T f_t \, dg_t
 &\weq (-1)^\theta \int_0^T D^{\theta}_{0+} (f-f(0+))(t) \, D^{1-\theta}_{T-} (g-g(T-))(t) \, dt \\
 &\qquad \qquad + f(0+)(g(T-)-g(0+)),
 %&\weq (-1)^\theta \int_a^b D^{\theta}_{a+} f_{a+}(t) \, D^{1-\theta}_{b-} g_{b-}(t) \, dt \\
 %&\qquad \qquad \qquad \qquad + f(a+)(g(b-)-g(a+)),
 \end{aligned}
\end{equation}
where the right side does not depend on $\theta$. In order to ensure the existence of the integral, we introduce the following spaces. For $\theta > 0$, we denote by $W_{\theta,1,T}(0+)$ the space of measurable functions $f: (0,T) \to \R$ such that $f(0+) \in \R$ exists and
\begin{equation}
\label{eq:norm}
 \Vert f\Vert_{\theta,1,T}
 = \int_0^T \frac{|f(t)|}{t^\theta} dt + \int_0^T \int_0^t \frac{|f(t)-f(s)|}{|t-s|^{1+\theta}} ds dt
\end{equation}
is finite. 
Similarly, we denote by $W_{\theta,\infty}(T-)$ the space of measurable functions $f: (0,T) \to \R$ such that $f(T-) \in \R$ exists and
\[
 ||f||_{\theta,\infty,T}
 \weq \sup_{t \in (0,T)} \frac{|f(T-) - f(t)|}{(T-t)^\theta} + \sup_{t \in (0,T)} \int_t^T\frac{|f(t)-f(s)|}{|t-s|^{1+\theta}} \, ds < \infty.
 \]
As $T$ is fixed, throughout the paper we drop the dependence on $T$ and simply write $W_{\theta,\infty}$ and $\Vert f\Vert_{\theta,\infty}$ instead of $W_{\theta,\infty}(T-)$ and $\Vert f\Vert_{\theta,\infty,T}$.  
We have the following result (see, e.g. \cite{Nualart_Rascanu_2002}).
\begin{proposition}
\label{the:ZSIntegralBound}
Assume that $f \in W_{\theta,1,T}(0+)$ and $g \in W_{1-\theta,\infty}$ for some $\theta \in (0,1)$. Then the integral in \eqref{eq:ZSIntegral} is well defined, representable as
\begin{equation}
 \label{eq:ZSIntegralSimple}
 \int_0^T f_t \, dg_t
 \weq (-1)^\theta \int_0^T D^{\theta}_{0+} f(t) \, D^{1-\theta}_{T-} (g-g(T-))(t) \, dt,
\end{equation}
and bounded by
\begin{equation}
 \label{eq:ZSIntegralBound}
 \left|\int_a^b f_t \, dg_t \right|
 \wle \frac{||f||_{\theta,1,T} \, ||g||_{1-\theta,\infty}}{\Gamma(\theta) \Gamma(1-\theta)}.
\end{equation}
In this case, for every $t\in[0,T]$ the restriction $\textbf{1}_{[0,t]}f$ belongs  to $W_{\theta,1,T}(0+)$ and 
the integral
$$
\int_0^t f_s \, dg_s = \int_0^T \textbf{1}_{[0,t]}(s)f_s \, dg_s
$$
is well-defined.
\end{proposition}
Motivated by this result, we introduce the space $W_{\theta,1}$ (which do depend on $T$ as well but omitted on the notation) as the space of functions such that \eqref{eq:norm} is finite, but $f(0+)$ does not necessarily exists. We use the following definition for our integral.
\begin{definition}
\label{def:integral}
Let  $f \in W_{\theta,1}$ and $g \in W_{1-\theta,\infty}$ for some $\theta \in (0,1)$. Then we define the integral by 
\begin{equation}
 \label{eq:ZSIntegralSimple-definition}
 \int_0^T f_t \, dg_t
 \weq (-1)^\theta \int_a^b D^{\theta}_{0+} f(t) \, D^{1-\theta}_{T-} (g-g(T-))(t) \, dt.
\end{equation}
\end{definition}
We get the following result stating that our integral is well-defined.
\begin{proposition}
Let  $f \in W_{\theta,1}$ and $g \in W_{1-\theta,\infty}$ for some $\theta \in (0,1)$. Then the integral \eqref{eq:ZSIntegralSimple-definition} is well-defined and bounded according to \eqref{eq:ZSIntegralBound}. Moreover, for every $t\in[0,T]$ the integral 
$$
\int_0^t f_s \,dg_s = \int_0^T \textbf{1}_{[0,t]}(s)f_s\,dg_s
$$
is well-defined.
\end{proposition}

Throughout the paper we consider integrals with respect to $g$ that is H\"older continuous of some order strictly larger than $1-\theta$ and $f\in W_{\theta,1}$ such that $f$ is bounded. Thus we introduce the following notation: the H\"older seminorm of order $\theta > 0$ of a measurable function $x: [0,T] \to \R$ is denoted by
\[
 [x]_{\theta,\infty}
 \weq \sup_{0 \le s < t \le T} \frac{|x(t) - x(s)|}{|t-s|^\theta},
\]
and the {\bf Gagliardo seminorm} of order $\theta>0$ and exponent $p$ by
\begin{equation}
 \label{eq:Gagliardo}
 [x]_{\theta, p}
 \weq \left(\int_0^T \int_0^T \frac{|x(t) - x(s)|^p}{|t-s|^{1+\theta p}} \, ds \, dt \right)^{\frac{1}{p}}.
\end{equation}
We will make use of the following simple proposition all the time.
\begin{proposition}
\label{prop:remove-first-term}
Let $f$ be bounded. Then $f\in W_{\theta,1}$ if and only if $[f]_{\theta,1}  <\infty$. Moreover, for every $f_n$ such that $f_n \to f$ pointwise and $f_n$ is uniformly bounded, we have $f_n \to f$ in $W_{\theta,1}$ if and only if $[f_n-f]_{\theta,1} \to 0$.
 \end{proposition}
\begin{proof}
Since $f$ is bounded, the first term in \eqref{eq:norm} is bounded. Thus the result follows from the very definitions of $W_{\theta,1}$ and $[f]_{\theta,1}$.
Similarly, the second assertion follows easily from Lebesgue dominated convergence theorem. Indeed, since $f_n$ and $f$ are bounded, it follows that 
$$
\int_0^T |f_n(t)-f(t)|t^{-\theta}dt \leq C\int_0^T t^{-\theta}dt < \infty,
$$
and thus pointwise convergence implies 
$$
\int_0^T |f_n(t)-f(t)|t^{-\theta}dt \to 0.
$$
Thus it suffices to consider only the second term in \eqref{eq:norm} which is $[\cdot]_{\theta,1}$.
\end{proof}
We also exploit the following proposition.
\begin{proposition}
\label{prop:forget-indicator}
Suppose that $f_n$ is uniformly bounded and $f_n \to 0$ pointwise. Set $f_{n,t} = f_n \textbf{1}_{\cdot \leq t}$. Then 
$[f_{n,t}]_{\theta,1} \to 0$ if and only if
$$
\int_0^t \int_0^t \frac{|f(s)-f(r)|}{|s-r|^{\theta+1}}dsdr \to 0.
$$
\end{proposition}
\begin{proof}
We have
$$
|f_{n,t}(s) - f_{n,t}(r)| \leq |f_n(s)-f_n(r)|\textbf{1}_{s\leq t} + |f_n(r)||\textbf{1}_{s\leq t} - \textbf{1}_{r\leq t}|.
$$
Now 
$$
|\textbf{1}_{s\leq t} - \textbf{1}_{r\leq t}| = \textbf{1}_{s\leq t<r} + \textbf{1}_{r\leq t<s}
$$
which is integrable with respect to $|s-r|^{-\theta-1}dsdr$. The claim follows from this.
\end{proof}
\section{Pathwise integrals of discontinuously evaluated stochastic processes}
\label{sec:integration}
In this section we briefly recall and refine essential results from \cite{chen-et-al} that ensures the existence of pathwise integrals of type
$$
\int_0^T \sigma(X_s)dY_s,
$$
where $\sigma$ is of locally bounded variation and $X$ and $Y$ are suitable processes. The essential differences in our case are;
\begin{enumerate}
\item By defining the pathwise integrals using \eqref{eq:ZSIntegralSimple-definition}, we can drop the assumption that $\sigma(X_{0+})$ exists. As $\sigma$ contains discontinuous, this fact is crucial in order to study general SDEs \eqref{eq:SDE}.
\item In \cite{chen-et-al} the authors assumed the existence of density function Lebesgue almost everywhere for $X$ such that it has an integrable upper bound (cf. Example \ref{example:density}). In our case, this assumption is usually not satisfied whenever $\sigma$ attains zero at some points (cf. Example \ref{example:nualart}). Thus we work with the essential Assumption \ref{assumption:key} directly (see also Remark 3.3 in \cite{chen-et-al}).
\end{enumerate}
We begin with the following Proposition, taken from \cite{chen-et-al}.

\begin{proposition}[\cite{chen-et-al}]
\label{prop:key-estimate}
Let $f: \R \to \R$ be right-continuous and of finite variation. Let $x: [0,T] \to \R$ be H\"older continuous of order $\alpha > 0$.
Then the \emph{Gagliardo seminorm} defined in (\ref{eq:Gagliardo}) of order $\theta  \in (0,1)$ and exponent $p \in [1,\infty)$ of the composite path $f \circ x$ is bounded by
%Then for any $\theta  \in (0,1)$ and $p \in [1,\infty)$, \changedi{the Gagliardo seminorm $[f \circ x]_{\theta ,p}$ of the composite path $f \circ x$} is bounded by
\begin{equation}
 \label{eq:GagliardoNonrandom}
  [f \circ x]_{\theta ,p}^p
  \wle 2^{p+1} (\theta  p)^{-1} \mu_f(K_x)^{p-1} [x]_{\alpha ,\infty}^{\theta  p/\alpha } \, \int_0^T  \int_{K_x} |x_t - y|^{-\theta  p/\alpha } \, \mu_f(dy) \, dt,
\end{equation}
where $K_x$ is the closure of the range of $x$.
\end{proposition}
The following theorem ensures the existence of the integral.
\begin{theorem}
\label{the:ZSIntegral}
Let $X$ and $Y$ be H\"older continuous random processes of orders $\alpha$ and $\eta$, respectively, such that $\alpha + \eta > 1$. Suppose that $X$ satisfies Assumption \ref{assumption:key}. Then for any $f: \R \to \R$ of locally finite variation, the pathwise integral
\[
 \int_0^T f(X_t) \, d Y_t
\]
exists almost surely in the sense of \eqref{eq:ZSIntegralSimple-definition}.
\end{theorem}
\begin{proof}
The claim follows by following the arguments of the proof of Theorem 3.1 in \cite{chen-et-al} and using Proposition \ref{prop:key-estimate}. Indeed, using Definition \ref{def:integral} allows us to drop the assumption that $f(X_{0+})$ exists, and by the localisation argument we may suppose that $\mu_f$ has compact support. Then the claim $f \circ x \in W_{1,\beta}$ follows from Proposition \ref{prop:remove-first-term} and Proposition \ref{prop:key-estimate} applied with $p=1$ and $\theta=\beta$ together with Assumption \ref{assumption:key}. 
\end{proof}

In order to extend several other key results of \cite{chen-et-al} we need the following simple lemma.
\begin{lemma}
\label{lma:very-simple}
Suppose $X$ satisfies Assumption \ref{assumption:key}. Then for any $\delta\in(0,\beta)$ we also have
$$
\sup_{y\in\R} \int_0^T \E|X_t-y|^{-\frac{\beta-\delta}{\alpha}}dt < \infty.
$$ 
\end{lemma}
\begin{proof}
The claim follows directly from the observation
$$
|X_t-y|^{-\frac{\beta-\delta}{\alpha}} = \textbf{1}_{|X_t-y|\geq 1}|X_t-y|^{-\frac{\beta-\delta}{\alpha}} + \textbf{1}_{|X_t-y|<1}|X_t-y|^{-\frac{\beta-\delta}{\alpha}} \leq 1 +|X_t-y|^{-\frac{\beta}{\alpha}}.
$$
\end{proof}
Using Lemma \ref{lma:very-simple} the following result follows essentially from the proof of Lemma A.3 in \cite{chen-et-al}. For this reason we present only the main differences.
\begin{lemma}
\label{lemma:MeanContinuity}
Let $X$ satisfy Assumption \ref{assumption:key}. Then for any $\theta \in (1-\beta,\beta)$ and $q \ge \theta/\alpha$, the functions
\[
 y \ \mapsto \
 \E \left\{ (1+[X]_{\alpha,\infty})^{-q} \int_0^T \int_0^t \frac{\textbf{1}_{X_s < y < X_t}}{(t-s)^{1+\theta}} \, ds dt \right\}
\]
and
\[
 y \ \mapsto \
 \E \left\{ (1+[X]_{\alpha,\infty})^{-q} \int_0^T \int_0^t \frac{\textbf{1}_{X_t < y < X_s}}{(t-s)^{1+\theta}} \, ds dt \right\}
\]
are bounded and continuous.
\end{lemma}
\begin{proof}
Denote $\phi(y) = \E \Phi(y)$, where
\[
 \Phi(y) \weq (1+[X]_{\alpha,\infty})^{-q} \int_0^T \int_0^t \frac{1_{(X_s,X_t)}(y)}{(t-s)^{1+\theta}} \, ds dt.
\]
Then 
\[
 \Phi(y)
 \wle \theta^{-1} \frac{[X]_{\alpha,\infty}^{\theta/\alpha}}{(1+[X]_{\alpha,\infty})^{q}} \, \int_0^T |X_t-y|^{-\theta/\alpha} \, dt
 \wle \theta^{-1} \int_0^T |X_t-y|^{-\theta/\alpha} \, dt.
\]
Thus using Lemma \ref{lma:very-simple} with $\delta = \beta- \theta$ we obtain that $\phi(y)$ is bounded. Similarly, we observe that $\phi(y)$ is right-continuous as long as we are able to show that for small enough $p>1$ we have
\begin{equation}
\label{eq:needed}
\sup_{\epsilon>0} \E \int_0^T \int_0^t \Phi_{1,\epsilon}(s,t)^p \, ds dt
< \infty,
\end{equation}
where
$$
 \Phi_{1,\epsilon}(s,t)
 \weq (1+[X]_{\alpha,\infty})^{-q} \, \frac{ \textbf{1}_{y \le X_s < y + \epsilon < X_t} }{(t-s)^{1+\theta}}.
$$
We choose $p \in (1, \frac{1+\beta}{1+\theta})$ so small that $1/p \ge 1 + \theta - \alpha q$.
Then
\[
 \Phi_{1,\epsilon}(s,t)^p
 \weq (1+[X]_{\alpha,\infty})^{-\tilde q} \, \frac{ \textbf{1}_{y \le X_s < y + \epsilon < X_t} }{(t-s)^{1+\tilde\theta}}
 \wle (1+[X]_{\alpha,\infty})^{-\tilde q} \, \frac{ \textbf{1}_{X_s < y + \epsilon < X_t} }{(t-s)^{1+\tilde\theta}},
\]
where $\tilde q = pq$ and $\tilde \theta = (1+\theta)p - 1$. Now our choice of $p$ implies that $\tilde q \ge \tilde \theta/\alpha$. Thus, as above, we obtain
$$
(1+[X]_{\alpha,\infty})^{-\tilde q} \, \int_0^T \int_0^t \frac{ \textbf{1}_{X_s < y + \epsilon < X_t} }{(t-s)^{1+\tilde\theta}}dsdt \leq \tilde\theta^{-1}\int_0^T |X_t-y-\epsilon|^{-\tilde\theta/\alpha}dt,
$$ 
from which \eqref{eq:needed} follows by noting that with our choice of $p$ we have $\tilde\theta < \beta$ and applying Lemma \ref{lma:very-simple}. The rest of the proof follows as in \cite{chen-et-al}.
\end{proof}

With Lemma \ref{lemma:MeanContinuity} at hand, the following result follows directly by following the proof of Proposition A.1 of \cite{chen-et-al}. For this reason, we omit the details.
\begin{proposition}
\label{prop:key-estimate-approx}
Suppose $X$ satisfies Assumption \ref{assumption:key} and let $f$ be of locally finite variation. Let $f_n$ be a standard smooth approximation of $f$. Then for any $\theta \in (0,\beta)$,
\[
 [f_n \circ X - f \circ X]_{\theta,1} \rightarrow 0.
\]
\end{proposition}

We also need the following result that provides us the solution candidate. The proof follows directly from the proof of Theorem 3.4 in \cite{chen-et-al} together with Proposition \ref{prop:key-estimate-approx}, and thus we omit the details.
\begin{theorem}
\label{the:Ito}
Assume that $X$ satisfies Assumption \ref{assumption:key} and let $f: \R \to \R$ be absolutely continuous, having a derivative $f'$ of locally finite variation. Then for any $t \in [0,T]$,
\begin{equation}
 \label{eq:Ito}
 f(X_t) - f(X_0)
 \weq \int_0^t f'(X_s) \, dX_s
\end{equation}
almost surely.
\end{theorem}

\section{Proofs of main results}
\label{sec:proofs}
We begin with the proof of Theorem \ref{thm:existence} that is now rather easy, taking account the results obtained in Section \ref{sec:integration}.
\subsection{Proof of Theorem \ref{thm:existence}}
We claim that $X_t = \Lambda^{-1}(\Lambda(X_0) + Y_t-Y_0)$ is a solution to \eqref{eq:SDE}. For this we observe first that
$$
\left[\Lambda^{-1}\right]'(y) =\sigma(\Lambda^{-1}(y)).
$$
As, by assumption, the process $Z_t= \Lambda(X_0)+Y_t-Y_0$ satisfies
$$
\sup_{z\in\R} \int_0^T |Z_t-z|^{-\frac{\beta}{\alpha}}dt < \infty,
$$
the claim follows by Theorem \ref{the:Ito} as long as we show that $y\mapsto \sigma(\Lambda^{-1})(y) \in BV_{loc}$. But this follows trivially from the fact that $\Lambda^{-1}(y)$ is an increasing function. Indeed, by \cite{josephy} a composition $\sigma \circ f$ on a compact interval is of bounded variation for any bounded variation function $\sigma$ if and only if there exists $N\in\mathbb{N}$ such that for all $a,b\in\R$ the pre-image $f^{-1}([a,b])$ can be represented as a union of $N$ intervals. Now, since $\Lambda^{-1}$ is increasing, the pre-image of $[a,b]$ under $\Lambda^{-1}$ is just $[\Lambda(a),\Lambda(b)]$. Hence $y\mapsto \sigma(\Lambda^{-1})(y) \in BV_{loc}$, and Theorem \ref{the:Ito} gives
$$
\Lambda^{-1}(\Lambda(X_0) + Y_t-Y_0) = \Lambda^{-1}(X_0) + \int_0^t \sigma(\Lambda^{-1}(\Lambda(X_0) + Y_s-Y_0))dY_s
$$
providing us one solution. This concludes the proof.

\subsection{Proof of Theorem \ref{thm:uniqueness} and Corollary \ref{cor:uniqueness}}
The idea of the proof of Theorem \ref{thm:uniqueness} is that we approximate $\sigma$ with non-negative smooth functions $\sigma_n$ such that $\frac{1}{\sigma_n}$ is locally integrable. With the help of these functions we then prove that, for any $\epsilon$, the solution is unique up to the first time when $\sigma(X_t)\leq \epsilon$. Then the uniqueness follows from the fact that $\epsilon$ is arbitrary. Moreover, by the standard localisation argument, instead of only locally bounded variation function we may and will assume without loss of generality, that $\sigma$ is of bounded variation. Indeed, since $X$ is H\"older continuous, it follows $\sigma(X)$ can be always identified with $\sigma_K(X)$, where the compact set $K =K_X$ is the closure of the range of $X$ and $\sigma_K$ is of bounded variation, and coincides with $\sigma$ on $K$. This fact will be used throughout this section without explicitly stated. We also recall that any bounded variation function $f$ can be represented as a difference of two increasing functions $f_+$ and $f_-$, i.e.
\begin{equation}
\label{eq:decomposition}
f = f_+ - f_-.
\end{equation}
We split the proof into several lemmas and propositions. 
\begin{lemma}
\label{lemma:approx}
Let $\sigma$ satisfying Assumption \ref{assu:sigma} with a decomposition
$$
\sigma = \sigma^+ - \sigma^-,
$$
where both $\sigma^+$ and $\sigma^-$ are increasing. Let $\xi$ be an arbitrary random variable with infinitely differentiable density function that has support on $[0,1]$, and define
$$
\sigma_n(x) = \E\sigma^+(x + n^{-1}\xi) - \E \sigma^-(x-n^{-1}\xi).
$$
Then $\sigma_n(x)$ is infinitely differentiable and converges pointwise to $\sigma$. Moreover, we have 
$$
\sigma_n^{-1}(x) \leq \sigma^{-1}(x)
$$
and
$$
\sigma_n(x) \leq |\mu|([x-1,x+1])+\sigma(x-1),
$$
where $|\mu|([a,b])$ is the variation of $\sigma$ over $[a,b]$. 
In particular, $\sigma_n^{-1}$ is locally integrable, and $\sigma_n(x)$ is bounded on compacts.
\end{lemma}
\begin{proof}
Since $\xi$ has infinitely differentiable density, it follows that $\sigma_n$ is infinitely differentiable. Moreover, the pointwise convergence follows from Lebesgue dominated convergence theorem. For the last two assertions, we first observe that since $\xi$ is supported on $[0,1]$
and $\sigma^+$ is increasing, we have
$$
\E\sigma^+(x + n^{-1}\xi) \geq \sigma^+(x).
$$
Similarly, 
$$
\E\sigma^-(x - n^{-1}\xi) \leq \sigma^-(x),
$$
and thus
$$
\sigma_n(x) \geq \sigma(x). 
$$
Since both $\sigma_n$ and $\sigma$ are non-negative, this implies
$$
0\leq \int_0^x\frac{1}{\sigma_n(y)}dy \leq \int_0^x\frac{1}{\sigma(y)}dy < \infty.
$$
Finally, since $\xi\in[0,1]$, we get
$$
\sigma_n(x) \leq \sigma_+(x+1) - \sigma_-(x-1).
$$
Finally, from 
$$
|\mu|([a,b])=\sigma_+(b)-\sigma_+(a) + \sigma_-(b)-\sigma_-(a)
$$
it follows that
$$
\sigma_+(x+1) -\sigma_-(x-1)= \sigma_+(x+1)-\sigma_+(x-1) + \sigma_+(x-1)-\sigma_-(x-1) \leq |\mu|([x-1,x+1])+\sigma(x-1).
$$
Since $\sigma$ is of locally bounded variation and locally bounded, this implies that $\sigma_n(x)$ is locally bounded as well.
\end{proof}
\begin{proposition}
\label{prop:lambda_n-rep}
Let $X$ be a solution to \eqref{eq:SDE} satisfying Assumption \ref{assumption:key} and for $\epsilon>0$, set
$$
\tau_{\epsilon} = \inf\{t : \sigma(X_s)\leq \epsilon\}.
$$
Let $\sigma_n$ be as in Lemma \ref{lemma:approx} and let $\Lambda_n(x) = \int_0^x \sigma_n^{-1}(y)dy$. Then 
for any $t\in[0,\tau_{\epsilon}]$, we have
$$
\Lambda_n(X_t) = \Lambda_n(X_0) + \int_0^t \frac{\sigma(X_s)}{\sigma_n(X_s)}d Y_s.
$$
\end{proposition}
\begin{proof}
By the definition of $\tau_\epsilon$ and Lemma \ref{lemma:approx}, we have, for any $s\in[0,\tau_\epsilon]$, that
$$
\frac{1}{\sigma_n(X_s)} \leq \frac{1}{\sigma(X_s)} \leq \epsilon^{-1}.
$$
This implies that 
$$
\left\vert \frac{1}{\sigma_n(X_s)} - \frac{1}{\sigma_n(X_r)}\right\vert \leq \epsilon^{-2}|\sigma_n(X_r)-\sigma_n(X_s)|,
$$
and thus
\begin{equation*}
\begin{split}
&\left\vert\frac{\sigma(X_s)}{\sigma_n(X_s)} - \frac{\sigma(X_r)}{\sigma_n(X_r)}\right\vert\\
& \leq \vert \sigma^{-1}_n(X_s)\vert \vert\sigma(X_s)-\sigma(X_r)\vert\\
&+ \vert \sigma(X_r)\vert \left\vert \frac{1}{\sigma_n(X_s)} - \frac{1}{\sigma_n(X_r)}\right\vert\\
&\leq \epsilon^{-1} \vert\sigma(X_s)-\sigma(X_r)\vert\\
&+ \vert \sigma(X_r)\vert \epsilon^{-2}\left\vert \sigma_n(X_s) - \sigma_n(X_r)\right\vert.
\end{split}
\end{equation*}
By Proposition \ref{prop:key-estimate} and Lemma \ref{lma:very-simple} we have, for $\theta\in(1-\alpha,\beta)$, that
$$
[\sigma \circ X]_{\theta,1} < \infty,
$$
i.e.
\begin{equation}
\label{eq:jelppi1}
\int_0^T \int_0^T \frac{|\sigma(X_r)-\sigma(X_s)|}{|r-s|^{\theta+1}}ds dr < \infty.
\end{equation}
Moreover, $\sigma \circ X$ is almost surely bounded. Thus, since $\sigma_n$ is Lipschitz continuous, we also have
$$
\int_0^T \int_0^T \frac{|\sigma(X_r)||\sigma_n(X_r)-\sigma_n(X_s)|}{|r-s|^{\theta+1}}ds dr \leq C\int_0^T \int_0^T \frac{|\sigma_n(X_r)-\sigma_n(X_s)|}{|r-s|^{\theta+1}}ds dr \leq C< \infty.
$$
Thus, by Proposition \ref{prop:remove-first-term},
$$
s \mapsto \frac{\sigma(X_s)\textbf{1}_{s\leq t}}{\sigma_n(X_s)} \in W_{\theta,1}.
$$
for every $t\leq \tau_\epsilon$. 
Note also that, by Lipschitz continuity of $\sigma_n$ and H\"older continuity of $X$, we have
$$
\Lambda_n(X_t) - \Lambda_n(X_0) = \int_0^t \frac{1}{\sigma_n(X_s)}dX_s.
$$
Here the integral exists in the sense of \ref{eq:ZSIntegralSimple-definition} as well as a Riemann--Stieltjes limit (see \cite{zahle99})
$$
\int_0^t \frac{1}{\sigma_n(X_s)}dX_s =\lim_n \sum_{k=1}^n \frac{X_{t_k}-X_{t_{k-1}}}{\sigma_n(X_{t_{k-1}})}.
$$

Therefore, denoting $\sigma_n^{\pi} (s) = \sum_{j=0}^{m-1} \frac{1_{[s_i,s_{i+1}]}(s)}{\sigma_n (x_{s_i})}$
and using that $x$ is a solution to \eqref{eq:SDE},
 we get
\begin{eqnarray}\label{eq21}
 \int_0^t \frac{d x_s}{\sigma_n (x_s)} &=& \lim\limits_{|\pi| \to 0} \sum_{j=0}^{m-1} \int_{s_i}^{s_{i+1}} \frac{\sigma(x_r)}{\sigma_n (x_{s_i})} d Y_r \notag\\
&=& \lim\limits_{|\pi| \to 0} \int^t_0 \sigma(x_r) \sigma_n^{\pi} (r) dY_r.
\end{eqnarray}
Thus it suffices to prove that, for any $t\leq \tau_\epsilon$, we have 
$$
\left\Vert \frac{\sigma(X_\cdot)\textbf{1}_{\cdot\leq t}}{\sigma_n(X_\cdot)} - \textbf{1}_{\cdot\leq t}\sigma(X_\cdot)\sigma_n^{\pi}(\cdot)\right\Vert_{\theta,1} \to 0.
$$
Moreover, in order to simplify the notation, Proposition \ref{prop:forget-indicator} implies that it suffices to consider integral over the region $0\leq s,r\leq t$ and drop the indicator term.
By Proposition \ref{prop:remove-first-term} together with the pointwise convergence and the fact that $\sigma(X_t)$ is almost surely bounded, it suffices to study the Gagliardo seminorm $[\cdot]_{\theta,1}$. We split
\begin{equation*}
\begin{split}
&\frac{\sigma(X_s)}{\sigma_n(X_s)} - \sigma(X_s)\sigma_n^{\pi}(s)-\frac{\sigma(X_r)}{\sigma_n(X_r)} - \sigma(X_r)\sigma_n^{\pi}(r)\\
&=\sigma(X_s)\left(\frac{1}{\sigma_n(X_s)}-\sigma_n^{\pi}(s)-\frac{1}{\sigma_n(X_r)}-\sigma_n^{\pi}(r)\right)\\
&+\left(\frac{1}{\sigma_n(X_r)}-\sigma_n^{\pi}(r)\right)\left(\sigma(X_s)-\sigma(X_r)\right).
\end{split}
\end{equation*}
Since $\sigma_n^{-1}(X_r) \leq \epsilon^{-1}$ for $r\in[0,\tau_\epsilon]$, it follows that  
$$
\left|\frac{1}{\sigma_n(X_r)}-\sigma_n^{\pi}(r)\right|\left|\sigma(X_s)-\sigma(X_r)\right| \leq 2\epsilon^{-1}\left|\sigma(X_s)-\sigma(X_r)\right|
$$
which is, by \eqref{eq:jelppi1}, integrable with respect to $|s-r|^{-\theta-1}ds dr$. Thus, again by Lebesgue dominated convergence Theorem, we have
$$
\int_0^t\int_0^t \left|\frac{1}{\sigma_n(X_r)}-\sigma_n^{\pi}(r)\right|\left|\sigma(X_s)-\sigma(X_r)\right||s-r|^{-\theta-1}drds \to 0.
$$
It remains to study the first term
$$
\int_0^t \int_0^t \left|\sigma(X_s)\left(\frac{1}{\sigma_n(X_s)}-\sigma_n^{\pi}(s)-\frac{1}{\sigma_n(X_r)}-\sigma_n^{\pi}(r)\right)\right||s-r|^{-\theta-1}dsdr.
$$
Now almost sure boundedness of $\sigma(X_s)$ implies
$$
|\sigma(X_s)|\left|\frac{1}{\sigma_n(X_s)}-\sigma_n^{\pi}(s)-\frac{1}{\sigma_n(X_r)}-\sigma_n^{\pi}(r)\right|\leq C\left|\frac{1}{\sigma_n(X_s)}-\sigma_n^{\pi}(s)-\frac{1}{\sigma_n(X_r)}-\sigma_n^{\pi}(r)\right|.
$$
To conclude, it was proved in \cite{zahle99} that for any H\"older continuous function $f$ we have
$$
[f-f_n]_{\theta,1} \to 0,
$$
where $f_n$ is the discrete approximation of $f$. Now $f_n=  \sigma_n^{\pi}(r)$ is a discrete approximation of H\"older continuous $f = \sigma^{-1}(X_s)$, and thus we observe that 
$$
\int_0^t \int_0^t \left|\frac{1}{\sigma_n(X_s)}-\sigma_n^{\pi}(s)-\frac{1}{\sigma_n(X_r)}-\sigma_n^{\pi}(r)\right||s-r|^{-\theta-1}ds dr \to 0.
$$
This concludes the proof.
\end{proof}
We also need the following elementary lemma.
\begin{lemma}
\label{lemma:UI}
Let $\mathcal{X}=(X,\mu,\Vert \cdot \Vert_\mu)$ be a normed space of functions. Let $f_n \in \mathcal{X}$ be a family of functions such that 
\begin{equation}
\label{eq:f_n-bound}
f_n \leq g_n + g,
\end{equation}
where $\Vert g_n\Vert_\mu \to 0$ and $g\in L^p(\mu)$ for some $p>1$. Then the family $f_n$ is uniformly integrable. In particular, if $f_n \to f \in \mathcal{X}$ pointwise, then 
$$
\Vert f_n-f\Vert_{\mu} \to 0.
$$ 
\end{lemma}
\begin{proof}
By the very definition of uniform integrability, we have to show that for each $\epsilon>0$ there exists $K$ such that 
$$
\sup_n\Vert f_n \textbf{1}_{f_n>K}\Vert_\mu < \epsilon.
$$
Let $\epsilon>0$ be arbitrary and let $K>0$ be a fixed number to be determined later. First we observe two elementary facts that, by \eqref{eq:f_n-bound}, we have
$$
\textbf{1}_{f_n>K} \leq \textbf{1}_{g_n+g>K}
$$
and 
$$
\textbf{1}_{g_n+g>K} \leq \textbf{1}_{g_n>\frac{K}{2}} + \textbf{1}_{g>\frac{K}{2}}.
$$
Thus, using \eqref{eq:f_n-bound} again, we get
\begin{equation*}
\begin{split}
f_n \textbf{1}_{f_n>K} &\leq (g_n+g)\textbf{1}_{f_n>K}\\
& \leq g_n + g \left(\textbf{1}_{g_n>\frac{K}{2}} + \textbf{1}_{g>\frac{K}{2}}\right). 
\end{split}
\end{equation*}
Triangle inequality gives us
\begin{equation*}
\begin{split}
\Vert f_n\textbf{1}_{f_n>K}\Vert_\mu &\leq \Vert g_n\Vert_\mu + \Vert g\textbf{1}_{g>\frac{K}{2}}\Vert_\mu+\Vert g\textbf{1}_{g_n>\frac{K}{2}}\Vert_\mu.
\end{split}
\end{equation*}
By H\"older and Chebyshev inequalities, we have, for conjugate $p,q$ such that $\frac{1}{p}+\frac{1}{q}=1$, that
$$
\Vert g\textbf{1}_{g_n>\frac{K}{2}}\Vert_\mu \leq \left( \int_X |g|^pd\mu\right)^{\frac{1}{p}} \left(\mu\left(g_n>\frac{K}{2}\right)\right)^{\frac{1}{q}} \leq 
\left( \int_X |g|^pd\mu\right)^{\frac{1}{p}}\Vert g_n\Vert_\mu^{\frac{1}{q}}\left(\frac{2}{K}\right)^{\frac{1}{q}}.
$$
Let now $N$ be large enough so that 
$$
\Vert g_n\Vert_\mu < \frac{\epsilon}{3}<1, \quad n\geq N.
$$
Then above yields, for all $n\geq N$, that
$$
\Vert f_n\textbf{1}_{f_n>K}\Vert_\mu < \frac{\epsilon}{3} + \Vert g\textbf{1}_{g>\frac{K}{2}}\Vert_\mu + \left( \int_X |g|^pd\mu\right)^{\frac{1}{p}}\left(\frac{2}{K}\right)^{\frac{1}{q}},
$$
and by choosing $K=\overline{K}$ such that 
$$
\Vert g\textbf{1}_{g>\frac{\overline{K}}{2}}\Vert_\mu < \frac{\epsilon}{3}
$$
and
$$
\left( \int_X |g|^pd\mu\right)^{\frac{1}{p}}\left(\frac{2}{\overline{K}}\right)^{\frac{1}{q}} <  \frac{\epsilon}{3}
$$
implies that 
$$
\Vert f_n\textbf{1}_{f_n>\overline{K}}\Vert_\mu <\epsilon
$$
for all $n\geq N$. Moreover, since $\Vert f_n\Vert_\mu < \infty$ for all $n$, it follows that for each $n$ there exists $K(n)$ such that 
$$
\Vert f_n\textbf{1}_{f_n>K(n)}\Vert_\mu < \epsilon.
$$
Thus it suffices to choose $K=\max(\overline{K},K(1),\ldots,K(N-1))$ to get
$$
\Vert f_n\textbf{1}_{f_n>K}\Vert_\mu < \epsilon.
$$
Finally, the last assertion follows from the well-known facts that pointwise convergence implies convergence in measure and convergence in measure together with the uniform integrability implies strong convergence in the norm $\Vert \cdot\Vert_\mu$.
\end{proof}
\begin{proposition}
\label{prop:lambda-rep}
Let $X$ be a solution to \eqref{eq:SDE} satisfying Assumption \ref{assumption:key} and $\tau_\epsilon$ be as in Proposition \ref{prop:lambda_n-rep}. Then for any $t\in[0,\tau_\epsilon]$ we have
$$
\Lambda(X_t) - \Lambda(X_0) = Y_t- Y_0.
$$
\end{proposition}
\begin{proof}
Let $\sigma_n$ be as in Lemma \ref{lemma:approx} and $\theta\in(1-\beta,\beta)$. As $\sigma_n$ converges almost everywhere to $\sigma$ and $\sigma_n^{-1} \leq \sigma^{-1}$, it follows from the dominated convergence theorem that 
$$
\Lambda_n(x) = \int_0^x \frac{1}{\sigma_n(y)}dy \to \int_0^x \frac{1}{\sigma(y)}dy = \Lambda(x).
$$
Now trivially 
$$
Y_t - Y_0= \int_0^t d Y_s,
$$
and by Proposition \ref{prop:lambda_n-rep} we have 
$$
\Lambda_n(X_t) - \Lambda_n(X_0) = \int_0^t \frac{\sigma(X_s)}{\sigma_n(X_s)}dY_s.
$$
Thus, using pointwise convergence $\sigma_n(x)\to \sigma(x)$ and Proposition \ref{prop:remove-first-term}, it suffices to prove that, for any $t\leq \tau_\epsilon$, we have
$$
\left[\frac{\sigma(X_\cdot)\textbf{1}_{\cdot \leq t}}{\sigma_n(X_\cdot)}-1 \right]_{\theta,1} \to 0.
$$
Moreover, again by Proposition \ref{prop:forget-indicator} it suffices to study the integral over the region $0\leq s,r\leq T$ and drop the indicator terms. 
We have, again by using the fact that $\sigma^{-1}_n(X_s)$, $\sigma_n(X_s)$, and $\sigma(X_s)$ are almost surely bounded by some (random) constant $C$, 
\begin{eqnarray}\label{eq24}
\left| \frac{\sigma(X_s)}{\sigma_n(X_s)} - \frac{\sigma(X_r)}{\sigma_n(X_r)} \right | &=& \left| \frac{ \sigma(X_s) \sigma_n(X_r) - \sigma(X_r) \sigma_n(X_s)}{\sigma_n(X_s) \sigma_n(X_r)} \right| \nonumber \\
& \le & C  | \sigma(X_s)\sigma_n(X_r) - \sigma(X_r) \sigma_n(X_s)  | \nonumber \\
& \le & C ( | \sigma(X_s) \big |  | \sigma_n(X_r) - \sigma_n(X_s)  | +    | \sigma_n(X_s) \big |  | \sigma(X_s) - \sigma(X_r)  |) \nonumber \\
& \le & C \left( | \sigma_n(X_r) - \sigma_n(X_s) \big | + |  \sigma(X_s) -  \sigma(X_r)  | \right) \nonumber \\
& \le & C \left(| \sigma_n(X_r) - \sigma_n(X_s) -\sigma(X_r)+\sigma(X_s)\big | + 2|  \sigma(X_s) -  \sigma(X_r)  | \right) \nonumber.
\end{eqnarray}
Let $d\mu=dsdr$. Choose next $p>1$ small enough such that $p\tilde\theta = p\theta +p - 1 \leq \beta$. 
By Proposition \ref{prop:key-estimate}, we have 
$$
\int_0^t \int_0^t |\sigma(X_s)-\sigma(X_r)|^p|s-r|^{-\theta p -p}dsdr \leq C \int_0^T |X_t- y|^{-\frac{p\tilde\theta}{\alpha}}dt,
$$
which is finite almost surely by Lemma \ref{lma:very-simple}. Thus
$g(s,r)=|\sigma(X_s)-\sigma(X_r)||s-r|^{-\theta-1} \in L^p(\mu)$ for our choice of $p$. Moreover, for 
$$
g_n(s,r) =\big| \sigma_n(X_r) - \sigma_n(X_s) -\sigma(X_r)+\sigma(X_s)\big ||s-r|^{-\theta-1}
$$
we have $\Vert g_n\Vert_\mu \to 0$ by Proposition \ref{prop:key-estimate-approx}. Thus using Lemma \ref{lemma:UI} with 
$$
f_n(s,r) = C^{-1}\left| \frac{\sigma(X_s)}{\sigma_n(X_s)} - \frac{\sigma(X_r)}{\sigma_n(X_r)} \right ||s-r|^{-\theta-1} 
$$
together with the fact $f_n \to 0$ pointwise we get
$$
\int_0^t \int_0^t C^{-1}_\sigma\left| \frac{\sigma(X_s)}{\sigma_n(X_s)} - \frac{\sigma(X_r)}{\sigma_n(X_r)} \right ||s-r|^{-\theta-1}ds dr \to 0
$$
which implies 
$$
\int_0^t \int_0^t \left| \frac{\sigma(X_s)}{\sigma_n(X_s)} - \frac{\sigma(X_r)}{\sigma_n(X_r)} \right ||s-r|^{-\theta-1}ds dr \to 0.
$$
This concludes the proof.
\end{proof}
We are now ready to prove Theorem \ref{thm:uniqueness}. 
\begin{proof}[Proof of Theorem \ref{thm:uniqueness}]
We begin by proving that $\tau$ is uniquely defined. 
Fix $\epsilon>0$ and let $X$ and $\tilde{X}$ be arbitrary solutions and let $\tau_\epsilon$ and $\tilde{\tau}_\epsilon$ be the corresponding stopping times defined in Proposition \ref{prop:lambda_n-rep}. Suppose $\tau_\epsilon < \tilde{\tau}_{\epsilon}$. Then, by Proposition \ref{prop:lambda-rep}, we have 
$$
\Lambda(X_t) - \Lambda(X_0) = Y_t = \Lambda(\tilde{X}_t)-\Lambda(X_0)
$$
on $t\in[0,\tau_\epsilon]$. Since $\Lambda$ has an inverse and $X$ and $\tilde{X}$ are H\"older continuous, it follows that actually $X=\tilde{X}$ on $t\in[0,\tau_\epsilon]$. Consequently, $\tilde{\tau}_{\epsilon} = \tau_\epsilon$ by the very definition. Furthermore, as, for any solution $X$, the mapping $\epsilon\mapsto \tau_\epsilon$ is decreasing, we obtain $\tau_\epsilon \to \tau$. Since for any $\epsilon>0$ the random time $\tau_\epsilon$ is uniquely defined, it follows that also $\tau$ is uniquely defined. Then H\"older continuity of $X$ implies that also the solution is unique up to $\tau$. Finally, the last assertion follows from the fact that $\tau =\infty$ whenever $\sigma(x)\neq 0$. 
\end{proof}
\begin{proof}[Proof of Corollary \ref{cor:uniqueness}]
Since $\sigma$ is of locally bounded variation and $Z_t = \Lambda^{-1}(\Lambda(X_0)+Y_t-Y_0)$ satisfies Assumption \ref{assumption:key}, it follows from Theorem \ref{the:Ito} together with the proof of Theorem \ref{thm:existence} that $Z$ is one solution. The uniqueness now follows from Theorem \ref{thm:uniqueness}.
\end{proof}
\begin{proof}[Proof of Corollary \ref{cor:uniqueness2}]
It is straightforward to check that if $Y_t$ has a density $p_t(y)$, then 
$$
Z_t = \Lambda^{-1}(\Lambda(X_0)+Y_t-Y_0)
$$
has a density $\tilde{p}_t(y)$ given by
$$
\tilde{p}_t(y) = p_t(\Lambda(y)-\Lambda(X_0)+Y_0)\Lambda'(y).
$$
Since $\Lambda'(y) = \frac{1}{\sigma(y)} \leq C$, we observe that 
$$
\sup_{y\in\R}\tilde{p}_t(y) \in L^1([0,T])
$$
and consequently, the solution $Z_t$ satisfies Assumption \ref{assumption:key}. The uniqueness then follows from Corollary \ref{cor:uniqueness}.
\end{proof}

\subsection*{Acknowledgements}

S. Torres is partially supported by the Project Fondecyt N. 1171335. L. Viitasaari wishes to thank Vilho, Yrj\"o, and Kalle V\"ais\"al\"a foundation for financial support.


\begin{thebibliography}{10}

\bibitem{bch05}
R.~Bass and Z-Q. Chen.
\newblock One-dimensional stochastic differential equations with singular and
  degenerate coefficients.
\newblock {\em The Indian Journal of Statistics}, 67(1):19--45, 2005.

\bibitem{BO}
B.~Boufoussi and Y.~Ouknine.
\newblock On a {SDE} driven by a fractional {B}rownian motion and with monotone
  drift.
\newblock {\em Electronic communications in probability}, 8(14):122--134, 2003.

\bibitem{chen-et-al}
Z.~Chen, L.~Leskel\"a, and L.~Viitasaari.
\newblock Pathwise stieltjes integrals of discontinuously evaluated stochastic
  processes.
\newblock {\em Stoch. Proc. Appl.}, 129(8):2723--2757, 2019.

\bibitem{ensc85-1}
H.~J. Engelbert and W.~Schmidt.
\newblock On one-dimensional stochastic differential equations with generalized
  drift.
\newblock {\em Lecture Notes in Control and Inform. Sci.}, 69:143--155, 1985.

\bibitem{ensc85}
H.~J. Engelbert and W.~Schmidt.
\newblock On solutions of one-dimensional stochastic differential equations
  without drift.
\newblock {\em Zeitschrift fr Wahrscheinlichkeitstheorie und Verwandte
  Gebiete.}, 68(3):287--314, 1985.

\bibitem{soledad-et-al}
J.~Garzon, J.A. Leon, and S.~Torres.
\newblock Fractional stochastic differential equation with discontinuous
  diffusion.
\newblock {\em Stochastic Analysis and Applications}, 35(6):1113--1123, 2017.

\bibitem{josephy}
M.~Josephy.
\newblock Composing functions of bounded variation.
\newblock {\em Proceedings of the {A}merican Mathematical Society},
  83(2):354--356, 1981.

\bibitem{legall83}
J.F. Le~Gall.
\newblock Local time applications to one-dimensional stochastic differential
  equations.
\newblock {\em Lecture Notes in Math.}, 986:15--31, 1983.

\bibitem{nualart-et-al}
J.A. Leon, D.~Nualart, and S.~Tindel.
\newblock Young differential equations with power type nonlinearities.
\newblock {\em Stoch. Proc. Appl.}, 127(9):3042--3067, 2017.

\bibitem{Mish}
Y.~Mishura.
\newblock {\em Stochastic calculus for fractional {B}rownian motion and related
  processes}.
\newblock Springer, Berlin, 2008.

\bibitem{MN}
Y.~Mishura and D.~Nualart.
\newblock Weak solutions for stochastic differential equations with additive
  fractional noise.
\newblock {\em Statistics and Probability Letters}, 70(4):253--261, 2004.

\bibitem{nakao72}
S.~Nakao.
\newblock On the pathwise uniqueness of solutions of one-dimensional stochastic
  differential equations.
\newblock {\em Osaka J. Math.}, 9:513--518, 1972.

\bibitem{Nualart_Rascanu_2002}
D.~Nualart and A.~R\u{a}\c{s}canu.
\newblock Differential equations driven by fractional {B}rownian motion.
\newblock {\em Collect. Math.}, 53:55--81, 2002.

\bibitem{Samko_Kilbas_Marichev_1993}
S.G. Samko, A.A. Kilbas, and O.I. Marichev.
\newblock {\em Fractional Integrals and Derivatives: Theory and Applications}.
\newblock Gordon and Breach Science Publishers, 1993.

\bibitem{veil-taqqu}
M.S. Veillette and M.S. Taqqu.
\newblock Properties and numerical evaluation of the {R}osenblatt distribution.
\newblock {\em Bernoulli}, 19(3):982--1005, 2013.

\bibitem{Zahle_1998}
M.~Z{\"a}hle.
\newblock Integration with respect to fractal functions and stochastic
  calculus. {I}.
\newblock {\em Probab. Theory Relat. Fields}, 111:333--372, 1998.

\bibitem{zahle99}
M.~Z\"ahle.
\newblock On the link between fractional and stochastic calculus.
\newblock {\em Stochastic Dynamics}, 2:305--325, 1999.

\end{thebibliography}
\end{document}